\documentclass[reqno,12pt]{amsart}
\usepackage{latexsym,amssymb,amsmath,epsfig}
\usepackage{amsthm}
\usepackage{enumerate}
\usepackage{mathrsfs}
\usepackage{graphicx}
\usepackage{pstricks-add}
\usepackage[latin1]{inputenc} % for Latex 2e
\newtheorem{theorem}{Theorem}[section]
\newtheorem{lemma}{Lemma}[section]
\newtheorem{corollary}{Corollary}[section]

\newtheorem{openquestion}{Open question}[section]
\begin{document}
\title{On the super edge-magicness of graphs of equal order and size}
\author{S. C. L\'opez}
\address{%
Departament de Matem\`{a}tica Aplicada IV\\
Universitat Polit\`{e}cnica de Catalunya. BarcelonaTech\\
C/Esteve Terrades 5\\
08860 Castelldefels, Spain}
\email{susana@ma4.upc.edu}

\author{F. A. Muntaner-Batle}
\address{Graph Theory and Applications Research Group \\
 School of Electrical Engineering and Computer Science\\
Faculty of Engineering and Built Environment\\
The University of Newcastle\\
NSW 2308
Australia}
\email{famb1es@yahoo.es}

\author{M. Prabu}
\address{British University Vietnam\\
Hanoi, Vietnam}
\email{mprabu201@gmail.com}

\maketitle

\begin{abstract}
The super edge-magicness of graphs of equal order and size has been shown to be important since such graphs can be used as seeds to answer many questions related to (super) edge-magic labelings and other types of well studied labelings, as for instance harmonious labelings. Also other questions related to the area of combinatorics can be attacked and understood from the point of view of super edge-magic graphs of equal order and size. For instance, the design of Steiner triple systems, the study of the set of dual shuffle primes and the Jacobsthal numbers. In this paper, we study the super edge-magic properties of some types of super edge-magic graphs of equal order and size, with the hope that they can be used later in the study of other related questions. The negative results found in last section are specially interesting since these kind of results are not common in the literature. Furthermore, the few results found in this direction usually meet one of the following reasons: too many vertices compared with the number of edges; too many edges compared with the number of vertices; or parity conditions. In this case, all previous reasons fail in our results.
\end{abstract}

\begin{quotation}

\end{quotation}

\section{Introduction} \label{Section Introduction}
All graphs contained in this paper may contain loops, however multiple edges are not allowed. Also, in order to make this paper reasonably self-contained, we mention that by a $(p,q)$-graph we mean a graph of order $p$ and size $q$. For integers $m\le n$, we use $[m,n]$ to denote $\{m,m+1,\ldots, n\}$.
In 1970, Kotzig and Rosa \cite{KotRos70} introduced the concepts of edge-magic graphs and edge-magic labelings as follows: Let $G$ be a $(p,q)$-graph. Then $G$ is called {\it edge-magic} if there is a bijective function $f:V(G)\cup E(G)\rightarrow [1,p+q]$ and a constant $k$ such that the sum $f(x)+f(xy)+f(y)=k$ for any $xy\in E(G)$. Such a function is called an {\it edge-magic labeling} of $G$ and $k$ is called the {\it valence} \cite{KotRos70} or the {\it magic sum} \cite{Wa} of the labeling $f$.

A restriction of the concept of edge-magic graphs and labelings was provided in 1998 by Enomoto et al. in \cite{E}. Let $f:V(G)\cup E(G) \rightarrow [1,p+q]$ be an edge-magic labeling of a $(p,q)$-graph G with the extra property that $f(V(G))= [1,p].$ Then G is called { \it super edge-magic} and $f$ is a {\it super edge-magic labeling} of $G$. Super edge-magic labelings of graphs of equal order and size have proven to be very important, since many relations with other problems as well as with other labelings have been established in the literature. For instance, super edge-magic labelings constitute a powerful link among other types of labelings \cite{F2,ILMR,LopMunRiu1,LopMunRiu8}. But also, (super) edge-magic labelings of two regular graphs have been proven to have many relations with other combinatorial problems. For instance, with Skolem and Langford type sequences \cite{LopMun15.a}, and hence with Steiner triple systems \cite{skolem} and also with dual shuffle primes and sequences of Jacobsthal numbers \cite{LopMunRiu13b}.

Many of the relations that have been enumerated above are possible thanks to the following product introduced originally in \cite{F1}, and that it is in fact a generalization of the Kronecker product (also known as tensor product, see \cite{HamImrKla} for other names) for digraphs. Let $D$ be a digraph and let $\Gamma$ be a family of digraphs with the same set $V$ of vertices. Assume that $h: E(D) \to \Gamma$ is any function that assigns elements of $\Gamma$ to the arcs of $D$. Then the digraph $D \otimes _{h} \Gamma$  is defined by (i) $V(D \otimes _{h} \Gamma)= V(D) \times V$ and (ii) $((a,i),(b,j)) \in E(D \otimes _{h} \Gamma) \Leftrightarrow (a,b) \in E(D)$ and $(i,j) \in E(h(a,b))$. Note that when $h$ is constant, $D \otimes _{h} \Gamma$ is the Kronecker product.

It is obvious that in order to use the $\otimes _{h}$-product, we need to deal with digraphs rather than with graphs. Therefore, the following definition that already appeared implicitly in \cite{F1} is necessary. We say that a digraph $D$ admits a labeling $\lambda$ if $und(D)$ admits $\lambda$, where $und(D)$ denotes the underlying graph of $D$.

The following results will be proven to be useful in the rest of the paper.

\begin{lemma}\label{super_consecutive} \cite{F2}
Let $G$ be a $(p,q)$-graph. Then $G$ is super edge-magic if and only if  there is a
bijective function $g:V(G)\to  [1,p]$ such
that the set $S=\{g(u)+g(v):uv\in E(G)\}$ is a set of $q$
consecutive integers.
%=\{s+i\}_{i=0}^{p-1}
\end{lemma}

The next result is a direct consequence of Lemma \ref{super_consecutive}.
\begin{lemma} \label{properties_SEM}
Let $D$ be a super edge-magic digraph in which each vertex is identified by the labels assigned by a super edge-magic labeling. Then the adjacency matrix $A(D)$ has the following properties.
\begin{enumerate}
    \item[(i)] Each counterdiagonal contains all 0's or all 0's except one 1.
    \item[(ii)] The set of counterdiagonals containing 1's in $A(D)$ is a set of consecutive diagonals.
\end{enumerate}
\end{lemma}

Although the definitions of (super)edge-magic graphs and the original Lemma \ref{super_consecutive} in \cite{F2} were established for simple graphs (that is to say, graphs without loops or multiple edges), it works exactly the same for graphs with loops.
From now on, whenever we talk about super edge-magic labelings in this paper, we will refer to labelings with the property provided in Lemma \ref{super_consecutive}, unless otherwise specified.

Let $f$ be %an edge-magic labeling of a $(p,q)$-graph $G$. The {\it complementary labeling} of $f$, denoted by $\overline{f}$, is the labeling defined by the rule: $\overline{f}(x)=p+q+1-f(x)$, for all $x\in V(G)\cup E(G)$. Notice that, if $f$ is an edge-magic labeling of $G$, we have that $\overline{f}$ is also an edge-magic labeling of $G$ with valence $\hbox{val}(\overline{f}) = 3(p+q+1)-\hbox{val}(f)$. In the case of
a super edge-magic labeling $f$ of a graph $G$. The \textit{super edge-magic complementary labeling}, $f^{c}$ is the labeling defined by the rule,
$ f^{c}(x) = p+1-f(x),$ for all $x\in V(G)$. Notice that, the labeling $f^c$ is also super edge-magic. The next lemma is an easy observation.

\begin{lemma} \label{complementary_property}
Let $D$ be a digraph and $f$ be a super edge-magic labeling of $D$. Let $A(D_f)$ denote the adjacency matrix of $D$ where each vertex takes the name of their labels in $f$. Then the matrix $A(D_{f^c})$ is a $\pi$ radians clockwise rotation of $A(D_f)$.
\end{lemma}

Next we need the following theorem.

\begin{theorem} \cite{LopMunRiu8} \label{otimes_product_G_and_snk}
Let $D$ be a (super) edge-magic digraph and let $S_n^k$ be the set of all super edge-magic labeled digraphs of order and size $n$ with minimum induced sum $k$. Assume that $h: E(D)\to S_n^k$. Then the graph $und(D \otimes_h S_n^k)$ is (super) edge-magic.
\end{theorem}

Theorem \ref{otimes_product_G_and_snk} allows us to use super edge-magic labeled (di)graphs of equal order and size as seeds in order to get new families of super edge-magic labeled (di)graphs.
We conclude this introduction by mentioning that super edge-magic graphs had been known since 1991 when Acharya and Hegde had introduced in \cite{AH} the concept of strongly indexable graph, which turns out to be equivalent to the concept of super edge-magic graph. For further information in graph labelings the interested reader can consult \cite{BaMi,G,MhMi,SlMb,Wa}.

The organization of the paper is as follows: we start by presenting different families of graphs with equal order and size which are super edge-magic, this is the content of Section \ref{section: familes SEM}. In Section \ref{section: familes NOT_SEM}, we consider few families of graphs with equal order and size and prove that they are not super edge-magic. Finally, we end by a short section of conclusion and open questions.

\section{Families of super edge-magic graphs of equal order and size}
\label{section: familes SEM}
We begin this section by providing some families of super edge-magic graphs of equal order and size. Then we will use these families in order to get other families of super edge-magic graphs. We denote by $LK_{1,n}$, the graph formed by a star $K_{1,n}$ with a loop attached at its central vertex. The vertex with the loop in $LK_{1,n}$ is called the central vertex.

\begin{theorem} \label{2LK(1,1)_union_LK(1,n)}
The graph $2LK_{1,1} \cup LK_{1,n}$ is super edge-magic for all $ n \in \mathbb{N}$.
\end{theorem}

\begin{proof}
Label the central vertex of one component $LK_{1,1}$ with $3$ and the other vertex of this component with $6$. Then label the central vertex of the other component isomorphic to $LK_{1,1}$ with $5$ and the other vertex vertex of this component with $2$. Finally, label the central vertex of the component isomorphic to $LK_{1,n}$ with $4$, and the vertices of degree 1 in $LK_{1,n}$ with the remaining labels in $[1,n+5]$. Then it is not hard to see that the resulting labeling is super edge-magic.\end{proof}

We illustrate the labeling of Theorem. \ref{2LK(1,1)_union_LK(1,n)} in Fig. \ref{Fig 1}.

\begin{figure}[h]
\begin{center}
\includegraphics {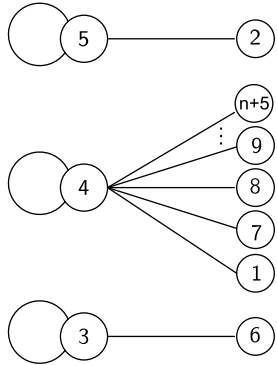}
\caption{A super edge-magic labeling of $2LK_{1,1} \cup LK_{1,n}$.}
\label{Fig 1}
\end{center}
\end{figure}

Theorem \ref{2LK(1,1)_union_LK(1,n)} can be generalized as follows.

\begin{theorem} \label{2LK(1,m)_union_LK(1,n)}
The graph $2LK_{1,m} \cup LK_{1,n}$ is super edge-magic for all $m,n \in \mathbb{N}$.
\end{theorem}

\begin{proof}
In order to label $2LK_{1,m} \cup LK_{1,n}$, consider the labeling of $2LK_{1,1} \cup LK_{1,n}$ obtained in the proof of Theorem \ref{2LK(1,1)_union_LK(1,n)}. Add $m-1$ new edges with their corresponding vertices of degree 1 attached to the central vertex of each of the two components that are originally isomorphic to $LK_{1,1}$. Label the new vertices of the component that has the central vertex labeled 5 with the numbers $-1,-3,-5,\ldots,-(2m-3)$. Label the remaining vertices (that is to say the new vertices in the component with the central vertex labeled 3) with the numbers $0,-2,-4,-6,\ldots,-(2m-4)$. Then it is easy to check that by adding $(2m-2)$ to each of the original labels, we obtain a super edge-magic labeling of the graph $2LK_{1,m} \cup LK_{1,n}$.
\end{proof}

We illustrate the procedure of the above proof in Fig. \ref{Fig 2LK13_LK14}.
\begin{figure}[h]
\begin{center}
\includegraphics {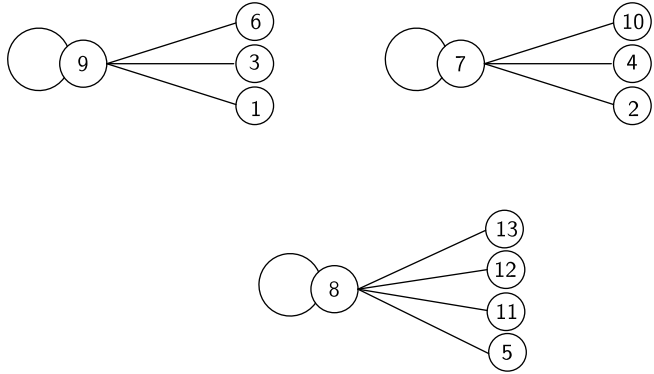}
\caption{A super edge-magic labeling of $2LK_{1,3} \cup LK_{1,4}$.}
\label{Fig 2LK13_LK14}
\end{center}
\end{figure}

Notice that, when $n=m$, we get that $3LK_{1,n}$ is super edge-magic for all $ n \in \mathbb{N}$. This fact can be generalized as follows.

\begin{theorem} \label{(2s+1)LK(1,n)SEM}
The graph $(2s+1)LK_{1,n}$ is super edge-magic for all $ n \in \mathbb{N}$ and $s \in \mathbb{N} \cup \{ 0 \}$.
\end{theorem}
\begin{proof}
It is easy check that any bijective function $f:V(LK_{1,n})\to [1,n+1]$ is super edge-magic, for all $n \in \mathbb{N}$ (see for instance, \cite{LMP1}). Hence, we can assume that $s \in \mathbb{N}$. Let us define the graph $(2s+1)LK_{1,n}$ as follows: $V((2s+1)LK_{1,n})= \{ v_i \}^{2s+1}_{i=1} \cup \{ v_i^j \}^{2s+1}_{i=1}, j= 1,2,\ldots,n$ and
$E((2s+1)LK_{1,n})= \{ v_iv_i \}^{2s+1}_{i=1} \cup \{ v_iv_i^j \}^{2s+1}_{i=1}, j= 1,2,\ldots,n$.
Next, we show that $(2s+1)LK_{1,n}$ is super edge-magic. Consider the labeling
$f: V((2s+1)LK_{1,n}) \to [1,(n+1)(2s+1)]$ defined by
$$f(v)=\left\{\begin{array}{ll}
s+i, & v=v_i,  i \in [1,2s+1],\\
 i-s-1 , &  v=v_i^1, i\in[ s+2, 2s+1],\\
f(v_i)+(2s+1), & v=v_i^1, i \in [1,s+1],\\
f(v_i^1)+(2s+1)(j-1), & v=v_i^j,  i \in [1,s+1] \ \mbox{and} \ j \neq 1, \\
f(v_i)+(2s+1)(j-1), & v=v_i^j,  i \in [s+2,2s+1] \ \mbox{and} \ j \neq 1.
\end{array}\right.$$
Then, $f$ is a super edge-magic labeling of $(2s+1)LK_{1,n}$ .
\end{proof}
\begin{figure}[h]
\begin{center}
\includegraphics {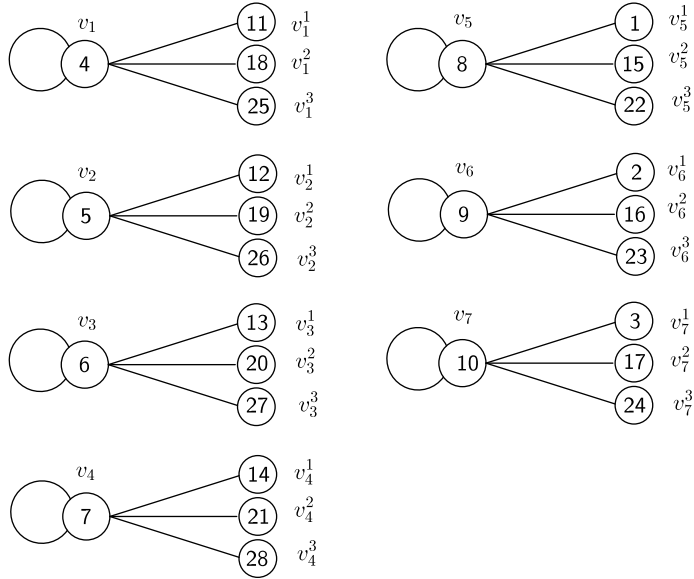}
\caption{A super edge-magic labeling of $7LK_{1,3}$.}
\label {Fig 7LK13}
\end{center}
\end{figure}
It is worth to mention that the labeling provided in the previous proof has been motivated by the technique introduced in \cite{LopMunRiu5} in order to prove that the crowns of certain cycles are perfect super edge-magic. In fact, when using this technique, there is one step in which super edge-magic labelings of $(2s+1)LK_{1,n}$ are obtained.
The next result that we want to consider is the following one.

\begin{theorem} \label{LK1,m_Union_2LK1,n_Union_2sLK1,1}
The graph $LK_{1,m} \cup 2LK_{1,n} \cup (2s)LK_{1,1}, \ s\in \mathbb{N}$ is super edge-magic.
\end{theorem}
\begin{proof}
Let us define the graph $G = LK_{1,m} \cup 2LK_{1,n} \cup (2s)LK_{1,1}, \ s\in \mathbb{N}$ as follows: $V(G)= \{v_i \}^{2s+3}_{i=1} \cup \{v^j_1\}^n_{j=1} \cup \{v^k_2\}^m_{k=1} \cup \{v^j_3\}^n_{j=1} \cup \{v^1_i \}^{2s+3}_{i=4}$ and $E(G)=\{v_iv_i\}^{2s+3}_{i=1} \cup \{v_1v^j_1\}^n_{j=1}\cup \{v_2v^k_2\}^m_{k=1}\cup \{v_3v^j_3\}^n_{j=1} \cup \{v_iv^1_i\}^{2s+3}_{i=4}$. Consider the labeling $f: V(G) \to [-2n+3,4s+m+5]$ defined by
$$f(v)=\left\{\begin{array}{ll}
2s+1+i, & v=v_i,  i \in [1,3],\\
s-2+i , &  v=v_i, i\in[4, s+3],\\
s+1+i , &  v=v_i, i\in[s+4, 2s+3],\\
f(v_i)+(2s+3), & v=v_i^1, i \in [1,2],\\
f(v_3)-(2s+3), & v=v_3^1,\\
-2j+4, & v=v^j_1,  j \in [2,n],\\
f(v_2^1)+k-1, & v=v_2^k,  k \in [2,m], \\
-2j+3, & v=v^j_3,  j \in [2,n],\\
f(v_i)+(2s+3), & v=v_i^1,  i \in [4,s+3],\\
f(v_i)-(2s+3), & v=v_i^1,  i \in [s+4,2s+3].
\end{array}\right.$$
By adding $(2n-2)$ to each of the original labels, we obtain a super edge-magic labeling of the graph $LK_{1,m} \cup 2LK_{1,n} \cup (2s)LK_{1,1}, \ s\in \mathbb{N}$.
\end{proof}

\begin{figure}[b]
\includegraphics[scale=0.8]{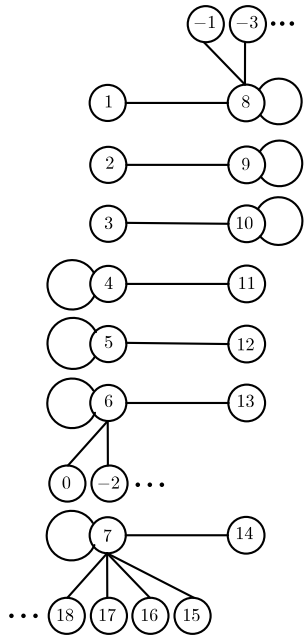}
\caption{A labeling pattern for $LK_{1,m} \cup 2LK_{1,n} \cup (2s)LK_{1,1}$.}
\label{Fig 5}
\end{figure}

Next, we introduce the concept of deer graph. Consider any caterpillar with an odd spine whose edges can be embedded in a horizontal line and the degree sequence of the vertices of the spine read the same from left to right than from right to left. If we attach a loop to the central vertex of the spine, we get a \textit{deer graph}. An example of a deer graph appears in Fig. \ref{Fig. deer}.

\begin{theorem} \label{deergraph}
All deer graphs are super edge-magic.
\end{theorem}

\begin{proof}
It suffices to label the vertices of the caterpillar in a traditional super edge-magic way.
\end{proof}

\begin{figure}[h]
\begin{center}
\includegraphics[scale=0.8]{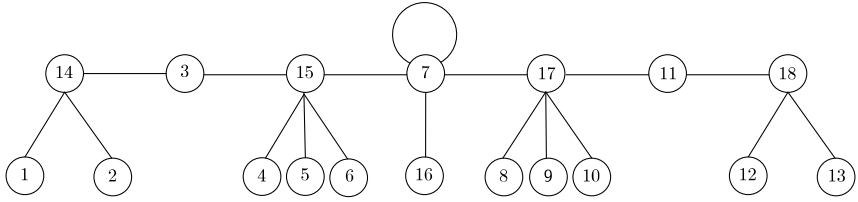}
\caption{A super edge-magic labeling of a deer graph with a spine of order 7.}
\label{Fig. deer}
\end{center}
\end{figure}
The {\it corona product} of two graphs  $G$ and $H$ is the graph $G \odot H$ obtained by placing a copy of $G$ and $|V(G)|$ copies of $H$ and then joining each vertex of $G$ with all vertices in one copy of $H$ in such a way that all vertices in the same copy of $H$ are joined exactly to one vertex of $G$. Let $\overline{K}_n$ be the complementary graph of the complete graph $K_n$, $n\in \mathbb{N}$.
Let $G$ be a graph with maximum degree $2$, by $\overrightarrow{G}$ we denote an orientation of $G$ in which all vertices have indegree $1$. The next lemma is an easy exercise.

\begin{lemma} \label{Rem_isomorphic}
If $C_k$ is a cycle with $k$ vertices, then $\hbox{und}(\overrightarrow{C_k} \otimes_h \overrightarrow{LK}_{1,n}) \cong C_k \odot \overline{K}_n$.
\end{lemma}

Combining Lemma \ref{Rem_isomorphic} with Theorems \ref{2LK(1,1)_union_LK(1,n)}, \ref{2LK(1,m)_union_LK(1,n)}, \ref{(2s+1)LK(1,n)SEM} and \ref{LK1,m_Union_2LK1,n_Union_2sLK1,1}, we get the following result.

\begin{corollary}
The following graphs are edge-magic.
\begin{enumerate}
\item[(i)] $2C_k \odot \overline{K}_1 \cup C_k \odot \overline{K}_n$,  for all $ k,n \in \mathbb{N}$,
\item[(ii)] $2C_k \odot \overline{K}_m \cup C_k \odot \overline{K}_n$,  for all $ k,m,n \in \mathbb{N}$
\item[(iii)]  $(2s+1)C_k \odot \overline{K}_n$, for all $ k,n \in \mathbb{N}$ and $s \in \mathbb{N} \cup \{ 0 \}$,
\item[(iv)] $C_k \odot \overline{K}_m \cup 2C_k \odot \overline{K}_n\cup (2s)C_k \odot \overline{K}_1,$ for all $k,s\in \mathbb{N}$.
\end{enumerate}
 In particular, if $k$ is odd, all the graphs above are super edge-magic.
\end{corollary}
\begin{proof}
We just prove (i), the other items can be proved similarly. By Lemma \ref{Rem_isomorphic}, $C_k \otimes_h (2LK_{1,1} \cup LK_{1,n}) \cong 2C_k \odot \overline{K}_1 \cup C_k \odot \overline{K}_n$. Combining this with Theorem \ref{2LK(1,1)_union_LK(1,n)} and Theorem \ref{otimes_product_G_and_snk}, $2C_k \odot \overline{K}_1 \cup C_k \odot \overline{K}_n$ is edge-magic. In particular, if $k$ is odd, $C_k$ is super edge-magic and hence the graph $2C_k \odot \overline{K}_1 \cup C_k \odot \overline{K}_n$ is super edge-magic.
\end{proof}

\section{Families of graphs of equal order and size which are not super edge-magic}
\label{section: familes NOT_SEM}

Gallian identifies in \cite{G} three usual reasons why graphs fail to admit labelings of certain types. The reasons are enumerated next.
\begin{enumerate}
  \item[(a)] Divisibility conditions.
  \item[(b)] Too many edges when we compare this number with the number of vertices.
  \item[(c)] Too many vertices when we compare this number with the number of edges.
\end{enumerate}

Our immediate goal is to prove a result about an infinite family of graphs that fails to be super edge-magic for different reasons than the ones enumerated above. We will prove that the all graphs of the family $LK_{1,m} \cup LK_{1,n}$ are not super edge-magic, for all positive integers $m$ and $n$. However, what we will really end up showing is that the digraph $D\cong \overrightarrow{LK}_{1,m} \cup \overrightarrow{LK}_{1,n}$ where $D$ is obtained by orienting the edges of ${LK}_{1,m} \cup {LK}_{1,n}$ in such a way that all vertices of degree 1 in ${LK}_{1,m} \cup{LK}_{1,n}$ have outdegree 0 in $D$ is not super edge-magic. We will do this using a contradiction argument.

Next, we describe some properties that the adjacency matrix of $D$ has.

\begin{lemma} \label{properties_of_digraphG}
Let $D$ be the digraph obtained from $LK_{1,m} \cup LK_{1,n}$, $m, n \in \mathbb{N}$, by orienting its edges in such a way that all vertices have indegree 1 in $D$. Then the adjacency matrix of $D$, denoted by $A(D)$ satisfies the following conditions:
\begin{enumerate}
  \item[(i)] Each column of $A(D)$  contains exactly one 1.
  \item[(ii)] All entries of $A(D)$ are either 0 or 1 and the entries 1 are located in exactly two rows of $A(D)$.
  \item[(iii)] Since each component of $D$ contains a loop, it follows that the two rows contain exactly one 1 in the main diagonal of $A(D)$.
  \end{enumerate}
\end{lemma}

Keeping all the above information in mind, we are now ready to state and prove the next result.

\begin{theorem} \label{LK(1,m)_union_LK(1,n)_notSEM}
The graph $LK_{1,m} \cup LK_{1,n}$ is not super edge-magic, for all positive integers $m$ and $n$.\end{theorem}

\begin{proof}
Assume to the contrary that $LK_{1,m} \cup LK_{1,n}$ is a super edge-magic graph and let $D$ be the digraph obtained from it by orienting its edges in such a way that all vertices in $D$ have indegree 1. By definition, $D$ is also super edge-magic. Let $f$ be a super edge-magic labeling of $D$, and let $A(D)=(a_{ij})$ be the adjacency matrix induced by $f$ where the vertices take the name of their labels in the super edge-magic labeling. By Lemma \ref{properties_of_digraphG}, all the entries 1 are located in exactly two rows of $A(D)$. Let these two row be row $i$ and row $j$ and assume that $i<j$.
If $a_{i1}=a_{i2}=\ldots=a_{il}=1$ and $a_{il+1}=0$, for some $l\ge 1$, then there is no $1$ in the diagonal with induced sum $i+l+1$, contradicting Lemma \ref{properties_SEM}.
Thus, we only have two possible generic forms for the adjacency matrix, either row $i$ is of the form
$(0  \ldots 0 1  \ldots 1 0 \ldots   \ldots 10 \ldots 0)$, with one more block of zeros than blocks of ones, or, $(0 \ldots 0 1  \ldots 10 \ldots\ldots 0  1\ldots 1)$, with exactly the same number of blocks of zeros and ones. Notice that, the first possibility is forbidden by Lemma \ref{complementary_property}, since otherwise, the adjacency matrix induced by the complementary labeling of $f^c$ would have exactly two rows with $1$ entries, namely $k,l$ with $k<l$, and row $k$ of the form $(1 \ldots 1 0  \ldots)$, which is a contradiction, as we have shown above. Hence, in what follows assume that row $i$  and row $j$ are of the form
$$(\overbrace{0 \ldots 0}^{|B_1|} T_1 \overbrace{0 \ldots 0}^{|B_2|} T_2\ldots \ldots 0T_k) \hbox{ and }(B_1 \overbrace{0 \ldots 0}^{|T_1|} B_2 \overbrace{0 \ldots 0}^{|T_2|}\ldots\ldots1 \overbrace{0 \ldots 0}^{|T_k|}),$$ respectively, where the $T_l$ and $B_l$ are blocks of $1$'s. By Lemma \ref{complementary_property}, we can assume that $\sum_{i=1}^k|T_i|=m+1$ and $\sum_{i=1}^k|B_i|=n+1$.

Notice that the block $B_1$ must be used to cover the zeros between $T_1$ and $T_2$, $B_2$ must be used to cover the zeros between $T_2$ and $T_3$ and so on. Since the length of $B_1$ must be equal to the number of zeros between $T_1$ and $T_2$ which is equal to the length of $B_2$ and so on,we get $\mid B_i \mid = (n+1)/k \ \hbox{for} \ i=1,2,\ldots, k$. A similar argument shows that $ \mid T_i \mid = (m+1)/k \ \hbox{for }  i=1,2,\ldots,k$. More over, since the first zero in the top row after $T_1$ appears in column $((m+n+2)/k)+1$, this can only be covered if the bottom row is $((m+n+2)/k)+i)$. This implies that
\begin{equation} \label{eq: differ_rows}
j-i = (m+n+2)/k
\end{equation}

On the other hand, the possible positions of $1$'s in the top row and bottom row are
$ (l'(m+n+2)- (m+1))/k + y$, $ l'=1,2, \ldots, k-1$, $y \in [1,(m+1)/k]$ and
$(l-1)(m+n+2))/k+w$, $l=1,2,\ldots, k-1$, $w \in [1,(n+1)/k],$ respectively. Since we assume that the graph is super edge-magic, each of the two rows contains a $1$ in the main diagonal. Thus, there exist $l,l',y$ and $w$ such that, $i=(l'(m+n+2)- (m+1))/k + y$ and $j=(l-1)(m+n+2))/k+w$. Hence, by(\ref{eq: differ_rows}) we obtain that $w=(2-l+l')(m+n+2)/k-(m+1)/k+y$. Since $l' \leq l$, $w$ is either negative or greater than $(n+1)/k$ which contradicts that $w \in [1,(n+1)/k]$.
\end{proof}

Let $L$ be the loop graph. Using a similar reasoning to the one used above, we are able to prove the following. result.

\begin{theorem}
The graph $L \cup LK_{1,n}$ is not super edge-magic for any $n \in \mathbb{N}$.
\end{theorem}
\begin{proof}
Suppose to the contrary that there exists $n \in \mathbb{N}$ such that $L \cup LK_{1,n}$ is super edge-magic and assume that each vertex is identified with the label assigned by a super edge-magic labeling. By definition, the digraph $D \cong \overrightarrow{L} \cup \overrightarrow{LK}_{1,n}$ obtained from $L \cup LK_{1,n}$ by orienting its edges in such a way that all vertices in $D$ have indegree $1$ in $D$ is also super edge-magic. The adjacency matrix of $D$ contains exactly two rows with $1$ entries, namely row $i$ and row $j$, $1\le i<j\le n+2$. If $a_{i1}=a_{i2}=\ldots=a_{il}=1$ and $a_{il+1}=0$, for some $l\ge 1$, then there is no $1$ in the diagonal with induced sum $i+l+1$, contradicting Lemma \ref{properties_SEM}.

Thus, we only have three possible generic forms for the adjacency matrix, either row $i$, $1\le i<n+2$, is of the form
a) $(0\ldots 0 1)$ or, b) $(0 1\ldots 1)$ or, c) $(0 \ldots 0 1 0 \ldots 0)$. Notice that, a) and b) are not possible since each of the two rows should contain a $1$ in the main diagonal, which is not possible in any of the two configurations. Finally, (c) is forbidden by Lemma \ref{complementary_property}, since otherwise, the adjacency matrix induced by the complementary labeling of $f^c$ would have exactly two rows with $1$ entries, namely $k,l$ with $k<l$, and row $k$ of the form $(1 \ldots 1 0 1 \ldots 1)$, which is a contradiction, as we have shown above.
\end{proof}

Out next immediate goal is to study the super edge-magicness of the graph $2L \cup LK_{1,n}$.

\begin{theorem} \label{2L_union_LK(1,n)_not_SEM}
The graph $2L\cup LK_{1,n}$ is not super edge-magic for any $n \in \mathbb{N}$.
\end{theorem}

\begin{proof}
Assume to the contrary that there exists a $n\in\mathbb{N}$ such that $2L \cup LK_{1,n}$ is super edge-magic. Following the same idea that we used in the two previous results, we consider the digraph $D \cong 2\overrightarrow{L} \cup \overrightarrow{LK}_{1,n}$ where the star is oriented as in the previous two proofs.  By definition, $2L \cup LK_{1,n}$ is super edge-magic if and only if the digraph $D \cong 2\overrightarrow{L} \cup \overrightarrow{LK}_{1,n}$ is super edge-magic. Let $f$ be a super edge-magic labeling of D, and let $A(D)=(a_{ij})$ be the adjacency matrix induced by $f$ where the vertices take the name of their labels in the super edge-magic labeling. Since all vertices of $D$ have indegree $1$, we can represent the adjacency matrix of $D$ by a vector $(v_1,\ldots, v_n)$ such that $v_k=i$ if and only if $a_{ik}=1$.
There are three possible generic forms.

{\it Case 1}: Let $v_1=v_2=\ldots=v_k=1$, $1 \leq k \leq n+1$ and $v_{k+1}>1$. Then there is no $1$ in the diagonal with induced sum $k+2$, contradicting Lemma \ref{properties_SEM}.

{\it Case 2}: Let $i$ be such that $a_{1i}=1$. By case 1, $i>1$. Let $v_1=v_2=\ldots=v_k=i$ and $v_{k+1} \neq i$ , $2 \leq k \leq n+1$ and $2 \leq i \leq n+2$. Since each row with entries being $1$ has a $1$ in the main diagonal and $i>1$, the other two rows of the adjacency matrix represent the two loops. Let $v_{k+1}$ and $v_l, k+2\leq l \leq n+3$ be the components that represent the two loops. Then, $\min\{v_{k+1},v_l\}>v_k$. If $v_l>v_{k+1}$, then there is no $1$ in the diagonal with induced sum $i+k+1$. If $v_l < v_{k+1}, v_l\neq i$, then $v_k < v_l < v_{k+1}, k+2\leq l \leq n+2$. This implies that one of the two rows representing the loop components cannot have a 1 in the main diagonal. Hence we get a contradiction in each of the above possible scenarios.

{\it Case 3}: Let $v_1=v_2=\ldots=v_k=n+3$, $1 \leq k \leq n+1$. Notice that this must be a part of the component $LK_{1,n}$ otherwise there is no $1$ in the main diagonal. In particular, this implies that $v_{n+3}=1$, which in view of Lemma \ref{complementary_property} and case 1, is also a contradiction.
\end{proof}

\section{Conclusions and open questions}
In this paper, we have proved the families $2LK_{1,1} \cup LK_{1,n}$, $2LK_{1,m} \cup LK_{1,n}$, $(2s+1)LK_{1,n}$, $LK_{1,m} \cup 2LK_{1,n} \cup (2s)LK_{1,1}$ $\forall m,n,s \in \mathbb{N}$ and deer graphs are super edge-magic. We have also proved the families $LK_{1,m} \cup LK_{1,n}$, $L \cup LK_{1,n}$, $2L \cup LK_{1,n}$ are not super edge-magic. From what we have seen so far, the following open questions raise naturally.

\begin{openquestion}
Characterize the set of super edge-magic graphs of the form $nL \cup nLK_{1,s}$.
\end{openquestion}

Next, we propose the most general open question, although we think that it may be a really hard question to solve.

\begin{openquestion}
Characterize the set of super edge-magic graphs whose components are isomorphic to loops and graphs that are stars with a loop attached at their central vertices.
\end{openquestion}

\begin{openquestion}
For which values of $s \in \mathbb{N}$ the graph $(2s)LK_{1,n}$ is super edge-magic?
\end{openquestion}
We feel that these graphs are of great importance  since they have equal order and size. Therefore, the $\otimes_h$-product can be applied to generate further families of (super) edge-magic graphs.

\end{document}